\documentclass[11pt, a4paper, reqno]{article}

\usepackage[margin=2.5cm]{geometry}

\usepackage[utf8]{inputenc}
\usepackage[T1]{fontenc}
\usepackage[english]{babel}

\usepackage{amsmath, amsthm, amssymb}

\usepackage{authblk}

\usepackage[colorlinks=true, linkcolor=blue, citecolor=blue, urlcolor=blue]{hyperref}

\usepackage{orcidlink}

\def\esup{\operatornamewithlimits{ess\,sup}}
\def\Id{\operatorname{I}}
\def\Ces{\operatorname{Ces}}
\def\LM{\operatorname{LM}}
\def\Cop{\operatorname{Cop}}
\def\dual{\,^{^{\bf c}}\!}

\def\loc{\operatorname{loc}}

\allowdisplaybreaks 

\numberwithin{equation}{section}

\newtheorem{theorem}{Theorem}[section]

\theoremstyle{definition}

\theoremstyle{remark}
\newtheorem{remark}[theorem]{Remark}

\begin{document}

\title{On Weighted Local Morrey-type and Complementary Local Morrey-Type Spaces}

\author[1]{Amiran Gogatishvili \orcidlink{0000-0003-3459-0355}}
\author[2]{Yoshihiro Sawano \orcidlink{0000-0003-2844-8053}}
\author[3]{Tu\u{g}\c{c}e \"{U}nver \orcidlink{0000-0003-0414-8400}}

\affil[1]{\small Institute of Mathematics of the Czech Academy of Sciences, \v Zitn\'a~25, 115~67 Praha~1, Czech Republic \protect\\ \texttt{gogatish@math.cas.cz}}
\affil[2]{\small Department of Mathematics, Chuo University, 1-13-27 Kasuga, Bunkyo-ku, Tokyo, 112-8551, Japan \protect\\ \texttt{yoshihiro-sawano@celery.ocn.ne.jp}}
\affil[3]{\small Department of Mathematics,
Kirikkale University,
71450 Yahsihan, Kirikkale,
T\"{u}rkiye}
\affil[*]{\small Corresponding author\; \texttt{tugceunver@kku.edu.tr} }

\renewcommand\Authands{ and }

\date{}

\maketitle

\noindent\textbf{This paper is dedicated to Professor Victor Burenkov on the occasion of his 85th birthday.}

\begin{abstract}
In this paper, we establish necessary and sufficient conditions for continuous embeddings between weighted local Morrey-type spaces and weighted complementary local Morrey-type spaces. Rather than relying on standard multidimensional approaches, our methodology is based on a one-dimensional reduction. Specifically, we reduce the multidimensional embedding problems to corresponding weighted inequalities involving one-dimensional Ces\`{a}ro and Copson function spaces. Applying this reduction alongside recent results in the characterization of the embeddings between Ces\`{a}ro and Copson spaces, we provide complete characterizations for the embeddings \[
\LM_{p_1,q_1}(v_1, w_1) \hookrightarrow \LM_{p_2,q_2}(v_2, w_2)\] 
and 
\[\dual\LM_{p_1,q_1}(v_1, w_1) \hookrightarrow \LM_{p_2,q_2}(v_2, w_2).\] 
Our results cover all possible finite cases of the parameters $0 < p_1, p_2, q_1, q_2 < \infty$.
\end{abstract}

\vspace{2mm}
\noindent \textbf{Keywords:} Local Morrey-type spaces, complementary local Morrey-type spaces, Ces\`{a}ro spaces, Copson spaces, embedding theorems, reduction theorem.

\vspace{2mm}
\noindent \textbf{2020 Mathematics Subject Classification:} Primary 46E30; Secondary 26D10,46E20, 47B38.

\section{Introduction}

The study of local Morrey-type spaces and their embedding properties plays a crucial role in harmonic analysis and the theory of function spaces. The main purpose of this paper is to characterize various embeddings between weighted local and complementary local Morrey-type spaces.

Before stating our main objectives, let us present some basic notations and definitions used throughout this paper. Consider a nonempty measurable set $A\subset \mathbb{R}^n$ where $n \ge 1$. Let $\mathcal{M}(A)$ and $\mathcal{M}^+(A)$ represent the classes of all measurable functions and all nonnegative measurable functions defined on $A$, respectively. A function defined on $A$ is called a weight if it is measurable, positive, and finite almost everywhere; the set of all such weights is denoted by $\mathcal{W}(A)$.

For $0 \in \mathbb{R}^n$ and $r > 0$, let $B(0,r)$ be the open ball centered at the origin of radius $r$ and $\text{dual } B(0,r) = \mathbb{R}^n \setminus B(0,r)$. Furthermore, for $s > t > 0$, we denote the annulus centered at the origin by $B(0,s) \setminus B(0,t) = \{y \in \mathbb{R}^n : t \le \vert{}y\vert{} < s\}$.

Let $X$ and $Y$ be quasi-normed vector spaces. We say that $X$ is continuously embedded into $Y$, denoted by $X \hookrightarrow Y$, if $X \subset Y$ and the identity operator $\Id : X \to Y$ is bounded. In this case, the norm of the embedding operator is defined by
\begin{equation*}
\|\Id\|_{X \to Y} = \sup_{z \in X \setminus \{0\}} \frac{\|z\|_Y}{\|z\|_X},
\end{equation*}
and the continuous embedding ensures that $\Vert{}z\Vert{}_Y \le \Vert{}\Id\Vert{}_{X \to Y} \Vert{}z\Vert{}_X$ for all $z \in X$. 

Let $p,q \in (0, \infty)$, $w \in \mathcal{W}(0,\infty)$, and $v \in \mathcal{W}(\mathbb{R}^n)$. We denote by $LM_{p,q}(v, w)$ the weighted local Morrey-type space, which consists of all functions $f \in L_{p,v}^{\loc}(\mathbb{R}^n)$ such that
\begin{equation*}
\|f\|_{LM_{p,q}(v, w)} = \bigg(\int_0^{\infty} \bigg( \int_{B(0,t)} |f(x)|^p v(x)^p dx \bigg)^{\frac{q}{p}} w(t)^q dt \bigg)^{\frac{1}{q}} < \infty.
\end{equation*}	
Analogously, $\dual LM_{p,q}(v, w)$ denotes the weighted complementary local Morrey-type space, defined as the collection of all functions $f \in L_{p, v}^{\loc}(\dual B(0, t))$, for every $t>0$, such that
\begin{equation*}
\|f\|_{\dual LM_{p,q}(v, w)} = \bigg(\int_0^{\infty} \bigg( \int_{\dual B(0,t)} |f(x)|^p v(x)^p dx \bigg)^{\frac{q}{p}} w(t)^q dt \bigg)^{\frac{1}{q}} < \infty.
\end{equation*}

Our main aim in this paper is to establish necessary and sufficient conditions for the following embeddings:
\begin{align}
\LM_{p_1,q_1}(v_1, w_1) &\hookrightarrow \LM_{p_2,q_2}(v_2, w_2) \label{LM-LM}\\
\LM_{p_1,q_1}(v_1, w_1) &\hookrightarrow \dual \LM_{p_2,q_2}(v_2, w_2) \label{LM-cLM}\\
\dual\LM_{p_1,q_1}(v_1, w_1) &\hookrightarrow  \LM_{p_2,q_2}(v_2, w_2) \label{cLM-LM}\\
\dual\LM_{p_1,q_1}(v_1, w_1) &\hookrightarrow  \dual\LM_{p_2,q_2}(v_2, w_2) \label{cLM-cLM}.
\end{align}

In recent decades, local Morrey-type spaces and the behavior of classical operators acting on them have been extensively investigated. A comprehensive overview of the development and fundamental results regarding these spaces is available in \cite{Burenkov-survey-1, Burenkov-survey-2}. For a broader perspective on the theory of Morrey spaces, we also refer the reader to the book by D.R. Adams \cite{Adams-book} and the recent two-volume monograph by Y. Sawano, G. Di Fazio, and D.I. Hakim \cite{SawFazHak-Vol1, SawFazHak-Vol2}. Alongside the study of the spaces themselves, the mapping properties of classical operators on these spaces have attracted considerable attention. The literature provides a thorough characterization of operator behavior within local Morrey-type spaces and their counterparts, namely, complementary local Morrey-type spaces. Necessary and sufficient conditions for the boundedness of maximal and fractional maximal operators, Riesz potentials, and singular integral operators within these spaces have been widely established; see, for instance, \cite{BurGul2004, BurGulGul2007-1, BurGulGul2007-2, BurGulTarSer, BurGul2009, BurGogGulMus, GKMT-CVEE, GKMT-JMA}.

The standard method for studying the boundedness of these operators in local Morrey-type spaces relies on H\"{o}lder's inequality. However, estimating via H\"{o}lder's inequality can lead to a loss of sharpness. Consequently, one can achieve more precise results for the boundedness problem by directly applying embedding characterizations between local Morrey-type spaces. Motivated by this, a series of previous works focused on the embedding problems \eqref{LM-LM}--\eqref{cLM-cLM}. In \cite{GMU-EMJ}, the embeddings \eqref{cLM-LM} and \eqref{LM-cLM} were investigated using a duality approach to convert the problem into iterated Hardy inequalities. Later, in \cite{GU-NS}, the embedding \eqref{LM-LM} was reduced to a one-dimensional problem between Ces\`{a}ro function spaces; by applying the one-dimensional characterizations from \cite{Unver}, partial results were obtained. 

A significant limitation in \cite{GU-NS} was that the reduced one-dimensional inequalities were only known under conditions where duality was applicable. Now that the complete characterizations of Ces\`{a}ro and Copson embeddings are available in \cite{GPU-JFA} and \cite{GU-AMP-Ces}, we utilize these unrestricted characterizations in the present work to completely overcome that limitation. Therefore, our approach in this paper relies on reducing the problem at hand to the embedding relations between Ces\`{a}ro and Copson spaces. These spaces can be regarded as the one-dimensional analogues of Morrey-type spaces for functions defined on $(0, \infty)$, and they are defined as follows: 
the weighted Cesàro space $\Ces_{p,q}(v,w)$ is the space of all measurable functions
$f$ on $(0,\infty)$ such that
\begin{equation}\label{eq:cesaro_def}
\|f\|_{\Ces_{p,q}(v,w)}
:=
\left(
\int_0^\infty
\left(
\int_0^t |f(s)|^p v(s)^p\,ds
\right)^{q/p}
w(t)^q\,dt
\right)^{1/q}
<\infty.
\end{equation}
Similarly, the weighted Copson space $\Cop_{p,q}(v,w)$ consists of all measurable
functions $f$ on $(0,\infty)$ satisfying
\begin{equation}\label{eq:copson_def}
\|f\|_{\Cop_{p,q}(v,w)}
:=
\left(
\int_0^\infty
\left(
\int_t^\infty |f(s)|^p v(s)^p\,ds
\right)^{q/p}
w(t)^q\,dt
\right)^{1/q}
<\infty.
\end{equation}

\begin{remark}
Write
\[
\widetilde{v_i}(s) = v_i\left(\frac{s}{|s|^2}\right)\frac{1}{|s|^{\frac{2n}{p_i}}}, \quad \text{and} \quad \widetilde{w_i}(\tau) = w_i\left(\frac{1}{\tau}\right) \frac{1}{\tau^{\frac{2}{q_i}}}
\]
for $i=1,2$.
By the change of variables 
\[
s=\frac{y}{|y|^2}\quad \text{and} \quad t=\frac{1}{\tau}
\]
\eqref{cLM-cLM} is equivalent to the embedding 
\[
\LM_{p_1,q_1}(\widetilde{v_1}, \widetilde{w_1}) \hookrightarrow \LM_{p_2,q_2}(\widetilde{v_2}, \widetilde{w_2}),
\]
and \eqref{LM-cLM} is equivalent to the embedding 
\[
\dual\LM_{p_1,q_1}(\widetilde{v_1}, \widetilde{w_1}) \hookrightarrow \LM_{p_2,q_2}(\widetilde{v_2}, \widetilde{w_2}).
\]
(cf. \cite[Remark~3.3]{GU-NS} and \cite{GMU-EMJ}). This note allows us to concentrate our attention on the characterizations of \eqref{LM-LM} and \eqref{cLM-LM}. Also note that the embeddings \eqref{LM-LM} and \eqref{cLM-LM} hold only for trivial functions if $p_2 > p_1$,
since in this case it is impossible to embed
$L^{p_1}(E)$ into $L^{p_2}(E)$
for any measurable set
$E$
of finite measure.
Therefore, throughout the paper, we will assume, without loss of generality, that $0 < p_2 \le p_1$ (cf. \cite{GMU-EMJ}).
\end{remark} 

\begin{remark}
It is known that the spaces $\LM_{p,q}(v, w)$ and $\dual \LM_{p,q}(v, w)$ are non-trivial, meaning they do not consist solely of functions equivalent to $0$ on $\mathbb{R}^n$, if and only if there exists some $t > 0$ such that
$$\int_t^{\infty} w(s)^q ds < \infty$$
and
$$\int_0^t w(s)^q ds < \infty,$$
respectively, see, for instance \cite{BurGul2004} and \cite{GMU-EMJ}.   
\end{remark}

Throughout the paper, the letters $c$ and $C$ denote generic positive constants that are independent of the main functions and parameters, whose values may change from line to line. Sub-scripted constants, such as $C_1$ or $\mathcal{C}_1$, remain fixed. We use the notation $A \lesssim B$ to state that $A \le \lambda B$ for some positive constant $\lambda$ independent of the essential quantities, and $A \approx B$ means $A \lesssim B \lesssim A$. In particular, when characterizing the norm of an embedding operator, $\Vert{}\Id\Vert{}_{X \to Y} \approx E$ asserts that the norm is bounded from above and below by $E$ up to multiplicative constants. For brevity, the differential elements (such as $dx$ or $dy$) in lengthy integral expressions are occasionally omitted.

The remainder of this paper is organized as follows. Section~\ref{section:Reduction Theorems} is devoted to the reduction theorems, where we investigate the fundamental relationship between embeddings among local Morrey-type spaces and corresponding embeddings among Ces\`{a}ro-type spaces. These reduction results serve as a crucial technical bridge for our subsequent analysis. In Section~\ref{section:Embeddings}, we present and prove our main results. By exploiting the reduction theorems established in Section~\ref{section:Reduction Theorems}, we successfully characterize the continuous embeddings $\LM_{p_1,q_1}(v_1, w_1) \hookrightarrow \LM_{p_2,q_2}(v_2, w_2)$ and $\dual\LM_{p_1,q_1}(v_1, w_1) \hookrightarrow \LM_{p_2,q_2}(v_2, w_2)$ (i.e., the embeddings \eqref{LM-LM} and \eqref{cLM-LM}).

\section{Reduction Theorems}
\label{section:Reduction Theorems}

As mentioned in the introduction, our approach relies on a one-dimensional reduction. In this section we reduce the embeddings between local Morrey-type spaces to the corresponding embeddings between Ces\`{a}ro-type spaces.

Let us briefly recall the polar coordinate formula in $\mathbb{R}^n$.
Denote the unit sphere by $S^{n-1}$. Every $x\in\mathbb{R}^n\setminus\{0\}$
can be written as
\[
x=rx', \qquad r=|x|,\quad x'=\frac{x}{|x|}\in S^{n-1}.
\]
Let $\sigma$ denote the unique Borel measure on $S^{n-1}$ such that
\begin{equation*}
\int_{\mathbb{R}^n} f(x)\,dx
=
\int_0^\infty
\left(
\int_{S^{n-1}} f(rx')\,d\sigma(x')
\right)
r^{n-1}\,dr
\end{equation*}
for every Borel measurable function $f:\mathbb{R}^n\to\mathbb{R}$ that is
either non-negative or integrable. We refer the reader to \cite{Folland} for
further details.

A version of the following theorem, which relates the embeddings between local Morrey-type spaces to the embeddings between Ces\`{a}ro spaces, was originally established in \cite[Theorem~2.1-2.2]{GU-NS}. Here, we reformulate it to suit the needs of the present work.

\begin{theorem} \label{T:Morrey-Cesaro}
Let $0 < p_2 \le p_1 < \infty$ and $0 < q_1, q_2 < \infty$. Assume that $v_1$ and $v_2$ are  weights on $\mathbb{R}^n$. Suppose that $w_1$ and $w_2$ are weights on $(0, \infty)$ such that $ \int_0^t w_i^{q_i} < \infty$, $i=1,2$
for some $t > 0$. 

Write
\begin{equation}\label{def-varphi}
\varphi(t) := 
\begin{cases} 
\displaystyle \left( \int_{S^{n-1}} \left[ v_1(ts')^{-1} v_2(ts') \right]^{\frac{p_1p_2}{p_1-p_2}} d\sigma(s') \right)^{\frac{p_1-p_2}{p_1}} t^{\frac{(n-1)(p_1-p_2)}{p_1}}, & \text{if } 0 < p_2 < p_1, \\[1.5em]
\displaystyle \esup_{s' \in S^{n-1}} \left[ v_1(ts')^{-1} v_2(ts') \right]^{p_1}, & \text{if } p_2 = p_1.
\end{cases}
\end{equation}
Then the following two statements are equivalent:

\textup{(i)} There exists a constant $C>0$ such that	
\begin{align}\label{Morrey-emb}
&\bigg(\int_0^{\infty} \bigg(\int_{B(0,t)} f(x)^{p_2} v_2(x)^{p_2} dx\bigg)^{\frac{q_2}{p_2}} w_2(t)^{q_2} dt 
\bigg)^{\frac{1}{q_2}} \notag \\ 
&\hskip+3cm\leq C \bigg( \int_0^{\infty} \bigg(\int_{B(0,t)} f(x)^{p_1} v_1(x)^{p_1} dx 
\bigg)^{\frac{q_1}{p_1}} w_1(t)^{q_1} dt \bigg)^{\frac{1}{q_1}}
\end{align}
holds for every $f\in \mathcal{M}^+(\mathbb{R}^n)$.

\textup{(ii)} There exists a constant $\mathcal{C}>0$ such that
\begin{equation}\label{Cesaro-emb}
\bigg(\int_0^{\infty} \bigg(\int_0^t g(s)^{p_2} \varphi(s) ds\bigg)^{\frac{q_2}{p_2}} w_2(t)^{q_2} dt 
\bigg)^{\frac{1}{q_2}}  \leq  \mathcal{C} \bigg( \int_0^{\infty} \bigg(\int_0^t g(s)^{p_1}  ds 
\bigg)^{\frac{q_1}{p_1}} 
w_1(t)^{q_1} dt \bigg)^{\frac{1}{q_1}}
\end{equation}
holds for every $g\in \mathcal{M}^+(0,\infty)$.

Moreover, the best constants $C$ and $\mathcal C$
of inequalities \eqref{Morrey-emb} and \eqref{Cesaro-emb}, respectively, satisfy $C \approx \mathcal{C}$.
\end{theorem}

\begin{proof}
First, observe that replacing the entire expression $(f v_1)^{p_1}$ directly with $f$ does not alter the set of admissible functions in \eqref{Morrey-emb}. If we elevate both sides of \eqref{Morrey-emb} to the power of $p_1$ and re-parameterize the indices as
\begin{equation}\label{eq:new_params}
    p = \frac{p_2}{p_1}, \quad q = \frac{q_2}{p_1}, \quad \theta = \frac{q_1}{p_1},
\end{equation}
while redefining the weights according to
\begin{equation}\label{eq:new_weights}
    u = w_2^{q_2}, \quad v = v_1^{-p_2} v_2^{p_2}, \quad w = w_1^{q_1},
\end{equation}
the original inequality reduces to the following equivalent form:
\begin{equation} \label{mor-ces-1}
\bigg(\int_0^{\infty} \bigg(\int_{B(0,t)} f(x)^{p} v(x) dx\bigg)^{\frac{q}{p}} u(t) dt 
\bigg)^{\frac{1}{q}}  \leq D \bigg( \int_0^{\infty} \bigg(\int_{B(0,t)} f(x) dx 
\bigg)^{\theta} w(t) dt \bigg)^{\frac{1}{\theta}},
\end{equation}
where $D$ is the best constant of \eqref{mor-ces-1}, and $D=C^{p_1}$.

Furthermore, by \cite[Theorem~2.1-2.2]{GU-NS}, inequality \eqref{mor-ces-1} holds if and only if the  inequality
\begin{equation}\label{reduced-ces}
\bigg(\int_0^{\infty} \bigg(\int_0^t g(s)^{p} v_p(s) ds\bigg)^{\frac{q}{p}} u(t) dt 
\bigg)^{\frac{1}{q}}  \leq  c \bigg( \int_0^{\infty} \bigg(\int_0^t g(s) ds 
\bigg)^{\theta} 
w(t)dt \bigg)^{\frac{1}{\theta}}
\end{equation}
holds for every $g\in \mathcal{M}^+(0,\infty)$,
where
\begin{equation} \label{tilde-v}
v_p(t):=
\begin{cases} 
\displaystyle \bigg(\int_{S^{n-1}} v(t s')^{\frac{1}{1-p}} d\sigma(s')\bigg)^{1-p} t^{(n-1)(1-p)}, & \text{if } 0 < p < 1, \\[1.5em]
\displaystyle \esup_{s' \in S^{n-1}} v(ts'), & \text{if } p = 1.
\end{cases}
\end{equation}
Importantly, the best constants of these two inequalities are equivalent, meaning $D \approx c$.

It remains to rewrite \eqref{reduced-ces} in the original parameters. Substituting back $p,q,\theta,u,v,w$ turns \eqref{reduced-ces} into
\begin{equation*} 
\bigg(\int_0^{\infty} \bigg(\int_0^t g(s)^{\frac{p_2}{p_1}} v_p(s) ds\bigg)^{\frac{q_2}{p_2}} w_2(t)^{q_2} dt 
\bigg)^{\frac{p_1}{q_2}} \leq c \bigg( \int_0^{\infty} \bigg(\int_0^t g(s) ds \bigg)^{\frac{q_1}{p_1}} w_1(t)^{q_1}dt \bigg)^{\frac{p_1}{q_1}},
\end{equation*}
(using $\varphi = v_p$ under this substitution, by \eqref{def-varphi} and \eqref{tilde-v}). 
Finally, by taking the $\frac{1}{p_1}$-th power of both sides and replacing $g$ by $g^{p_1}$, we arrive exactly at \eqref{Cesaro-emb}. The constant on the right-hand side becomes $c^{\frac{1}{p_1}}$. By definition, the best constant for this final inequality is $\mathcal{C}$, which implies that $\mathcal{C} = c^{\frac{1}{p_1}}$. Since $c\approx D = C^{p_1}$ , it immediately follows that $C \approx \mathcal{C}$. This completes the proof.
\end{proof}

In the following theorem, we present a reduction result that transforms the embedding between local Morrey-type and complementary local Morrey-type spaces into an embedding between Copson and Ces\`{a}ro spaces.

\begin{theorem} \label{T:cLM-LM-Copson-Cesaro}
Let $0 < p_2 \le p_1 < \infty$ and $0 < q_1, q_2 < \infty$. Assume that $v_1$ and $v_2$ are  weights on $\mathbb{R}^n$. Suppose that $w_1$ and $w_2$ are weights on $(0, \infty)$ such that $\int_0^t w_1^{q_1} < \infty$ and $ \int_t^{\infty} w_2^{q_2} < \infty$ 
for some $t > 0$. Then the following two statements are equivalent:

\textup{(i)} There exists a constant $C>0$ such that	
\begin{align}\label{cLM-LM-emb}
&\bigg(\int_0^{\infty} \bigg(\int_{B(0,t)} f(x)^{p_2} v_2(x)^{p_2} dx\bigg)^{\frac{q_2}{p_2}} w_2(t)^{q_2} dt 
\bigg)^{\frac{1}{q_2}} \notag \\ 
&\hskip+3cm\leq C \bigg( \int_0^{\infty} \bigg(\int_{\dual B(0,t)} f(x)^{p_1} v_1(x)^{p_1} dx 
\bigg)^{\frac{q_1}{p_1}} w_1(t)^{q_1} dt \bigg)^{\frac{1}{q_1}}
\end{align}
holds for every $f\in \mathcal{M}^+(\mathbb{R}^n)$.

\textup{(ii)} There exists a constant $\mathcal{C}>0$ such that
\begin{equation}\label{Copson-Cesaro-emb}
\bigg(\int_0^{\infty} \bigg(\int_0^t g(s)^{p_2} \varphi(s) ds\bigg)^{\frac{q_2}{p_2}} w_2(t)^{q_2} dt 
\bigg)^{\frac{1}{q_2}}  \leq  \mathcal{C} \bigg( \int_0^{\infty} \bigg(\int_t^{\infty} g(s)^{p_1}  ds 
\bigg)^{\frac{q_1}{p_1}} 
w_1(t)^{q_1} dt \bigg)^{\frac{1}{q_1}}
\end{equation}
holds for every $g\in \mathcal{M}^+(0,\infty)$, where $\varphi$ is defined in \eqref{def-varphi}.

Moreover, the best constants of inequalities \eqref{cLM-LM-emb} and \eqref{Copson-Cesaro-emb}, respectively, satisfy $C=\mathcal{C}$.
\end{theorem}

\begin{proof}
By the polar coordinate formula in $\mathbb{R}^n$, the inner integrals in
the definitions of the weighted Morrey-type spaces can be rewritten as
follows:
\begin{align}\label{LH-Morrey} 
\bigg(\int_0^{\infty} &\bigg(\int_{B(0,t)} f(x)^{p_2} v_2(x)^{p_2} dx\bigg)^{\frac{q_2}{p_2}} w_2(t)^{q_2} dt \bigg)^{\frac{1}{q_2}} \notag\\
& =  \bigg(\int_0^{\infty} \bigg(\int_0^t \bigg[\int_{S^{n-1}} f(rx')^{p_2} v_2(rx')^{p_2} d\sigma(x') \bigg] r^{n-1} dr \bigg)^{\frac{q_2}{p_2}} w_2(t)^{q_2} dt \bigg)^{\frac{1}{q_2}}
\end{align}
and
\begin{align}\label{RH-cMorrey} 
\bigg( \int_0^{\infty}& \bigg(\int_{\dual B(0,t)} f(x)^{p_1} v_1(x)^{p_1} dx \bigg)^{\frac{q_1}{p_1}} w_1(t)^{q_1} dt \bigg)^{\frac{1}{q_1}}\notag  \\
& =  \bigg(\int_0^{\infty} \bigg(\int_t^{\infty} \bigg[\int_{S^{n-1}} f(rx')^{p_1} v_1(rx')^{p_1} d\sigma(x') \bigg] r^{n-1} dr \bigg)^{\frac{q_1}{p_1}} w_1(t)^{q_1} dt \bigg)^{\frac{1}{q_1}}. 
\end{align}

$(i)\Rightarrow (ii)$:  Assume that \eqref{cLM-LM-emb} holds for all $f \in \mathcal{M}^+(\mathbb{R}^n)$.  Our goal is to show that \eqref{Copson-Cesaro-emb} holds for all $g\in \mathcal{M}^+(0,\infty)$. We will treat the cases $p_2 < p_1$ and $p_2 = p_1$ separately.

Let us begin with the case $0 < p_2 < p_1$.  For an arbitrary function $g \in \mathcal{M}^+(0,\infty)$, consider the test function $f$ defined by
$$
f(x) = g(|x|) \left( \int_{S^{n-1}}\left[ v_1(|x|\tau')^{-1} v_2(|x|\tau')\right]^{\frac{p_1p_2}{p_1-p_2}} \, d\sigma(\tau') \right)^{-\frac{1}{p_1}}v_1(x)^{-\frac{p_1}{p_1-p_2}}v_2(x)^{\frac{p_2}{p_1-p_2}} |x|^{\frac{1-n}{p_1}}.
$$ 
Using  \eqref{LH-Morrey} and \eqref{RH-cMorrey}
as well as \eqref{def-varphi}, we deduce
\begin{align}\label{LH-Morrey-LH-Cesaro} \bigg(\int_0^{\infty} \bigg(\int_{B(0,t)} f(x)^{p_2} v_2(x)^{p_2} dx\bigg)^{\frac{q_2}{p_2}} w_2(t)^{q_2} dt \bigg)^{\frac{1}{q_2}} &= \bigg(\int_0^{\infty} \bigg(\int_0^t g(r)^{p_2}  \varphi(r) dr\bigg)^{\frac{q_2}{p_2}} w_2(t)^{q_2} dt \bigg)^{\frac{1}{q_2}}
\end{align}
and
\begin{align}\label{RH-cMorrey-RH-Copson}
\bigg( \int_0^{\infty} &\bigg(\int_{\dual B(0,t)} f(x)^{p_1} v_1(x)^{p_1} dx \bigg)^{\frac{q_1}{p_1}} w_1(t)^{q_1} dt \bigg)^{\frac{1}{q_1}}  = \bigg(\int_0^{\infty} \bigg(\int_t^{\infty} g(r)^{p_1} dr \bigg)^{\frac{q_1}{p_1}} w_1(t)^{q_1} dt  \bigg)^{\frac{1}{q_1}}.  
\end{align}
Thus, by combining \eqref{LH-Morrey-LH-Cesaro} and \eqref{RH-cMorrey-RH-Copson}, and using assumption \eqref{cLM-LM-emb}, 
when $p_2<p_1$,
we conclude that \eqref{Copson-Cesaro-emb} is valid with the best constant $\mathcal{C}$ satisfying $\mathcal{C} \leq C$.

Now, assume that $p_1=p_2=p$. Let 
$\alpha \in (0,1)$ be fixed
and define
\[ E(r) = \left \{ x' \in S^{n-1}:  \frac{v_2(rx')}{v_1(rx')}\ge 
\alpha\esup_{\tau \in S^{n-1}} \frac{v_2(r\tau)}{v_1(r \tau)}\right \}= \left \{ x' \in S^{n-1}:  \frac{v_2(rx')}{v_1(rx')}\ge \alpha
\varphi(r)^{\frac1p}
\right \}\] 
for $r>0$. Note that $|E(r)|>0$ by definition.
For an arbitrary function $g \in \mathcal{M}^+(0,\infty)$, define
the test function $f$ by
$$
f(x) = g(|x|)|E(|x|)|^{-\frac{1}{p}}\chi_{E(|x|)}\left(\frac{x}{|x|}\right)
v_1(x)^{-1}|x|^{\frac{1-n}{p}}
$$
for $x \in {\mathbb R}^n \setminus\{0\}$.
Using \eqref{LH-Morrey}, we have for $t \in (0,\infty)$ 
\begin{align}
\bigg(\int_0^{\infty} & \bigg(\int_{B(0,t)} f(x)^{p} v_2(x)^{p} dx\bigg)^{\frac{q_2}{p}} w_2(t)^{q_2} dt \bigg)^{\frac{1}{q_2}}\notag\\
&= \bigg(\int_0^{\infty} \bigg(\int_0^t g(r)^p |E(r)|^{-1} \bigg[ \int_{E(r)} \left(\frac{v_2(rx')}{v_1(rx')}\right)^p  d\sigma(x') \bigg] dr\bigg)^{\frac{q_2}{p}} w_2(t)^{q_2} dt \bigg)^{\frac{1}{q_2}} \notag\\
&\ge \bigg(\int_0^{\infty} \bigg(\int_0^t g(r)^p |E(r)|^{-1} \bigg[ \int_{E(r)} 
\alpha^p\left( \esup_{\tau \in S^{n-1}} \frac{v_2(r\tau)}{v_1(r\tau)} \right)^p  d\sigma(x') \bigg] dr\bigg)^{\frac{q_2}{p}} w_2(t)^{q_2} dt \bigg)^{\frac{1}{q_2}} \notag\\
&=\alpha \bigg(\int_0^{\infty} \bigg(  \int_0^t g(r)^p \varphi(r)  dr\bigg)^{\frac{q_2}{p}} w_2(t)^{q_2} dt \bigg)^{\frac{1}{q_2}}. \label{Mor-p=1-1}
\end{align}
Meanwhile in view of \eqref{RH-cMorrey}, we have for $t\in(0,\infty)$
\begin{align}
&\bigg(\int_0^{\infty} \bigg(\int_{\dual B(0,t)} f(x)^pv_1(x)^p\, dx \bigg)^{\frac{q_1}{p}} w_1(t)^{q_1} dt \bigg)^{\frac{1}{q_1}}\notag\\
& \hskip+0.2cm= \bigg(\int_0^{\infty} \bigg(\int_t^\infty g(r)^p |E(r)|^{-1} r^{1-n} \bigg[ \int_{E(r)} v_1(rx')^{-p} v_1(r x')^p d\sigma(x') \bigg]  r^{n-1} dr \bigg)^{\frac{q_1}{p}} w_1(t)^{q_1} dt \bigg)^{\frac{1}{q_1}}\notag\\
&\hskip+0.2cm= \bigg(\int_0^{\infty} \bigg(\int_t^\infty g(r)^p \,dr \bigg)^{\frac{q_1}{p}} w_1(t)^{q_1} dt \bigg)^{\frac{1}{q_1}}. \label{cMor-p=1-1}
\end{align} 
Therefore,  by
taking into account the arbitrariness of $\alpha$
and
substituting \eqref{Mor-p=1-1} and \eqref{cMor-p=1-1} into our initial assumption \eqref{cLM-LM-emb}, we conclude that \eqref{Copson-Cesaro-emb} holds for every $g \in \mathcal{M}^+(0,\infty)$ with the best constant $\mathcal{C}$ satisfying $\mathcal{C} \leq C$.

$(ii)\Rightarrow (i)$: Assume now that \eqref{Copson-Cesaro-emb} holds for all $g \in \mathcal{M}^+(0, \infty)$. Our goal is to show that \eqref{cLM-LM-emb} holds for all $f \in \mathcal{M}^+(\mathbb{R}^n)$. As in the previous implication, we will treat the cases $0 < p_2 < p_1$ and $p_2 = p_1$ separately.

Let us first consider the case $0 < p_2 < p_1$. For an arbitrary $f \in \mathcal{M}^+(\mathbb{R}^n)$, utilizing \eqref{LH-Morrey} and applying H\"{o}lder's inequality with exponents $(\frac{p_1}{p_2}, \frac{p_1}{p_1-p_2})$, we obtain
\begin{align*}
\bigg(\int_0^{\infty} &\bigg(\int_{B(0,t)} f(x)^{p_2} v_2(x)^{p_2} dx\bigg)^{\frac{q_2}{p_2}} w_2(t)^{q_2} dt \bigg)^{\frac{1}{q_2}} \notag\\
&\leq \bigg(\int_0^{\infty} \bigg(\int_0^t \bigg[\int_{S^{n-1}}
f(rx')^{p_1} v_1(rx')^{p_1} d\sigma(x')\bigg]^{\frac{p_2}{p_1}} \notag\\
&\hskip+1cm \times \bigg[\int_{S^{n-1}} \left[v_1(rx')^{-1}v_2(rx')\right]^{\frac{p_1p_2}{p_1-p_2}} d\sigma(x')\bigg]^{\frac{p_1-p_2}{p_1}}
r^{n-1} dr \bigg)^{\frac{q_2}{p_2}} w_2(t)^{q_2} dt \bigg)^{\frac{1}{q_2}} \notag \\
&=  \bigg(\int_0^{\infty} \bigg(\int_0^t \varphi(r) \bigg[\int_{S^{n-1}}
f(rx')^{p_1} v_1(rx')^{p_1} d\sigma(x')\bigg]^{\frac{p_2}{p_1}} 
r^{\frac{p_2(n-1)}{p_1}} dr \bigg)^{\frac{q_2}{p_2}} w_2(t)^{q_2} dt \bigg)^{\frac{1}{q_2}}. 
\end{align*}
Now, applying \eqref{Copson-Cesaro-emb} for
\[
g(r)= \bigg[\int_{S^{n-1}}
f(rx')^{p_1} v_1(rx')^{p_1} d\sigma(x')\bigg]^{\frac{1}{p_1}} 
r^{\frac{n-1}{p_1}}
\]
and taking equality
\eqref{RH-cMorrey} into account, we arrive at
\begin{align*}
\bigg(\int_0^{\infty} &\bigg(\int_{B(0,t)} f(x)^{p_2} v_2(x)^{p_2} dx\bigg)^{\frac{q_2}{p_2}} w_2(t)^{q_2} dt \bigg)^{\frac{1}{q_2}} \notag\\
&  \leq \mathcal{C} \bigg(\int_0^{\infty} \bigg(\int_t^\infty \bigg[\int_{S^{n-1}}
f(rx')^{p_1}v_1(rx')^{p_1} d\sigma(x')\bigg]  r^{n-1} dr \bigg)^{\frac{q_1}{p_1}} w_1(t)^{q_1} dt \bigg)^{\frac{1}{q_1}}\\
&= \mathcal{C} \bigg( \int_0^{\infty} \bigg(\int_{\dual B(0,t)} f(x)^{p_1}v_1(x)^{p_1}  dx
\bigg)^{\frac{q_1}{p_1}} w_1(t)^{q_1} dt \bigg)^{\frac{1}{q_1}}.
\end{align*}
Hence, \eqref{cLM-LM-emb} holds. Moreover, the best constant $C$ in this inequality satisfies $C \leq \mathcal{C}$.

The proof for the case $p_1 = p_2 = p$ is completely analogous, where the application of H\"{o}lder's inequality is replaced by taking the essential supremum (corresponding to conjugate exponents $1$ and $\infty$). In view of \eqref{LH-Morrey}, we have
\begin{align*}
\bigg(\int_0^{\infty} & \bigg(\int_{B(0,t)} f(x)^{p} v_2(x)^{p} dx\bigg)^{\frac{q_2}{p}} w_2(t)^{q_2} dt \bigg)^{\frac{1}{q_2}}\notag\\
&\leq \bigg(\int_0^{\infty} \bigg(\int_0^t \bigg[\int_{S^{n-1}}
f(rx')^{p} v_1(rx')^{p} d\sigma(x')\bigg] \notag\\
&\hskip+3cm \times \esup_{x'\in S^{n-1}} \left[v_1(rx')^{-1} v_2(rx')\right]^p  r^{n-1} dr \bigg)^{\frac{q_2}{p}} w_2(t)^{q_2} dt \bigg)^{\frac{1}{q_2}} \notag \\
& =  \bigg(\int_0^{\infty} \bigg(\int_0^t  \varphi(r) \bigg[\int_{S^{n-1}} f(rx')^{p} v_1(r x')^{p}  d\sigma(x')
\bigg] r^{n-1} dr \bigg)^{\frac{q_2}{p}} w_2(t)^{q_2} dt \bigg)^{\frac{1}{q_2}}.
\end{align*}
Applying inequality \eqref{Copson-Cesaro-emb} for
\[
g(r) = \bigg[\int_{S^{n-1}}
f(rx')^{p} v_1(rx')^{p} d\sigma(x')\bigg]^{\frac{1}{p}} 
r^{\frac{n-1}{p}}
\]
and using \eqref{RH-cMorrey}, we obtain that
\begin{align*}
\bigg(\int_0^{\infty} & \bigg(\int_{B(0,t)} f(x)^{p} v_2(x)^{p} dx\bigg)^{\frac{q_2}{p}} w_2(t)^{q_2} dt \bigg)^{\frac{1}{q_2}}\notag\\
& \leq \mathcal{C} \bigg(\int_0^{\infty} \bigg(\int_t^\infty \bigg[\int_{S^{n-1}}
f(rx')^{p} v_1(rx')^{p} d\sigma(x')\bigg]  r^{n-1} dr \bigg)^{\frac{q_1}{p}} w_1(t)^{q_1} dt \bigg)^{\frac{1}{q_1}}\\
&= \mathcal{C} \bigg( \int_0^{\infty} \bigg(\int_{\dual B(0,t)} f(x)^{p}v_1(x)^{p}  dx
\bigg)^{\frac{q_1}{p}} w_1(t)^{q_1} dt \bigg)^{\frac{1}{q_1}}.
\end{align*}
Hence, \eqref{cLM-LM-emb} holds. Moreover, the best constant $C$ in this inequality satisfies $C \leq \mathcal{C}$.
This completes the proof.
\end{proof}

\section{Embeddings} \label{section:Embeddings}

The aim of this section is to characterize the embeddings \eqref{LM-LM} and \eqref{cLM-LM}.

\subsection{\texorpdfstring{$ \LM_{p_1,q_1}(v_1, w_1) \hookrightarrow \LM_{p_2,q_2}(v_2, w_2)$}{LM(p1,q1) to LM(p2,q2)}}

Denote by ${\rm I}$
the embedding operator
from
$ \LM_{p_1,q_1}(v_1, w_1) $
into
$\LM_{p_2,q_2}(v_2, w_2)$
if it exists and define
\[
\|\Id\|_{\LM_{p_1,q_1}(v_1, w_1) \rightarrow \LM_{p_2,q_2}(v_2, w_2)}:=\sup_{f\in  \LM_{p_1,q_1}(v_1, w_1)\setminus \{0\} }\frac{ \|f\|_{ \LM_{p_2,q_2}(v_2, w_2)}}{\|f\|_{\LM_{p_1,q_1}(v_1, w_1)}} .\]

\begin{theorem}\label{T:LM-LM}
Let $0 < p_1, p_2, q_1, q_2 < \infty$. Assume that $w_1$ and $w_2$ are weights on $(0, \infty)$ such that $\int_t^{\infty} w_i^{q_i} < \infty$, $i=1,2$ for all $t \in (0, \infty)$ and $v_1$ and $v_2$ are  weights on $\mathbb{R}^n$.

{\rm(i)} If  $q_1\le p_2< p_1 \le q_2$, then \eqref{LM-LM} holds for all $f \in \mathcal{M}^+(\mathbb{R}^n)$ if and only if 
\begin{equation}\label{C1}
C_1 :=  \esup_{t\in (0,\infty)}
\left[\bigg(\int_t^{\infty} w_1^{q_1} \bigg)^{-\frac{1}{q_1}} \esup_{s\in (t,\infty)} \bigg( \int_s^{\infty} w_2^{q_2} \bigg)^{\frac{1}{q_2}}  \bigg(\int_{B(0,s)\setminus B(0,t)} \left[ v_1^{-1} v_2 \right]^{\frac{p_1p_2}{p_1-p_2}} \bigg)^{\frac{p_1-p_2}{p_1p_2}}
\right]
\end{equation}
is finite.
Moreover, 
\[
\|\Id\|_{\LM_{p_1,q_1}(v_1, w_1) \rightarrow \LM_{p_2,q_2}(v_2, w_2)} \approx C_1.
\]

{\rm(ii)} If  $q_1\le q_2 < p_1$, and $q_1 \le p_2 < p_1$ then \eqref{LM-LM} holds for all $f \in \mathcal{M}^+(\mathbb{R}^n)$ if and only if 
\begin{align}
C_2 :=  \sup_{t\in (0,\infty)} &\bigg(\int_t^{\infty} w_1^{q_1} \bigg)^{-\frac{1}{q_1}} \notag \\
&\times  \bigg[\int_t^{\infty} \bigg( \int_s^{\infty}w_2^{q_2} \bigg)^{\frac{q_2}{p_1-q_2}} w_2(s)^{q_2}\bigg(\int_{B(0,s)\setminus B(0,t)} \left[ v_1^{-1} v_2 \right]^{\frac{p_1p_2}{p_1-p_2}} \bigg)^{\frac{(p_1-p_2)q_2}{(p_1-q_2)p_2}}\,ds \bigg]^{\frac{p_1-q_2}{p_1q_2}} \label{C2}
\end{align}
is finite.
Moreover, 
\[
\|\Id\|_{\LM_{p_1,q_1}(v_1, w_1) \rightarrow \LM_{p_2,q_2}(v_2, w_2)} \approx C_2.
\]

{\rm(iii)} If  $p_2 < q_1 \le q_2$, and $p_2 < p_1 \le q_2$ then \eqref{LM-LM} holds for all $f \in \mathcal{M}^+(\mathbb{R}^n)$ if and only if $C_1 < \infty$, where $C_1$ is defined in \eqref{C1} and
\begin{align}
C_3 :=  \sup_{t \in (0,\infty)} &\bigg(\int_t^{\infty} w_2^{q_2} \bigg)^{\frac{1}{q_2}}  \bigg[\int_0^{t} \bigg(\int_s^{\infty} w_1^{q_1}\bigg)^{-\frac{q_1}{q_1-p_2}} w_1(s)^{q_1}\notag \\
&\times\bigg(\int_{B(0,t)\setminus B(0,s)} \left[ v_1^{-1} v_2 \right]^{\frac{p_1p_2}{p_1-p_2}} \bigg)^{\frac{(p_1-p_2)q_1}{(q_1-p_2)p_1}}\,ds  \bigg]^{\frac{q_1-p_2}{q_1p_2}} \label{C3}
\end{align}
is finite.
Moreover, 
\[
\|\Id\|_{\LM_{p_1,q_1}(v_1, w_1) \rightarrow \LM_{p_2,q_2}(v_2, w_2)} \approx C_1+C_3.
\]

{\rm(iv)} If $p_2 < q_1 \le q_2 < p_1$
then \eqref{LM-LM} holds for all $f \in \mathcal{M}^+(\mathbb{R}^n)$ if and only if
$C_2 < \infty$ and $C_3< \infty$, where $C_2$ and $C_3$ are defined in \eqref{C2} and \eqref{C3}, respectively. Moreover, 
\[
\|\Id\|_{\LM_{p_1,q_1}(v_1, w_1) \rightarrow \LM_{p_2,q_2}(v_2, w_2)} \approx C_2 + C_3.
\]

{\rm(v)} If $q_2 < q_1 \leq p_2 < p_1$ then \eqref{LM-LM} holds for all $f \in \mathcal{M}^+(\mathbb{R}^n)$ if and only if 
\begin{align*}
C_4 := \bigg(\int_0^{\infty}  &\bigg(\int_t^{\infty} w_2^{q_2} \bigg)^{\frac{q_2}{q_1-q_2}} w_2(t)^{q_2}  \esup_{s\in (0, t)}  \bigg(\int_s^{\infty} w_1^{q_1}\bigg)^{-\frac{q_2}{q_1-q_2}} \notag \\
&\times\bigg(\int_{B(0,t)\setminus B(0,s)} \left[ v_1^{-1} v_2 \right]^{\frac{p_1p_2}{p_1-p_2}} \bigg)^{\frac{(p_1-p_2)q_1q_2}{(q_1-q_2)p_1p_2}}\,dt  \bigg)^{\frac{q_1-q_2}{q_1q_2}} 
\end{align*}
and
\begin{align}\label{C5}
C_5 &:=  \bigg\{\int_0^{\infty} w_1(t)^{q_1} \esup_{y\in (0,t)} \bigg(\int_y^{\infty}w_1^{q_1}\bigg)^{-\frac{q_1}{q_1-q_2}}\bigg[ \int_{y}^t \bigg( \int_s^t  w_2^{q_2}\bigg)^{\frac{q_2}{p_1-q_2}} w_2(s)^{q_2} \notag \\
&  \hspace{2cm} \times \bigg(\int_{B(0,s)\setminus B(0,y)} \left[ v_1^{-1} v_2 \right]^{\frac{p_1p_2}{p_1-p_2}}
 \bigg)^{\frac{(p_1-p_2)q_2}{(p_1-q_2)p_2}} ds \bigg]^{\frac{(p_1-q_2)q_1}{(q_1-q_2)p_1}} dt \bigg\}^{\frac{q_1-q_2}{q_1q_2}}
\end{align}
are finite.
Moreover, 
\[
\|\Id\|_{\LM_{p_1,q_1}(v_1, w_1) \rightarrow \LM_{p_2,q_2}(v_2, w_2)} \approx C_4+C_5.
\]

{\rm(vi)} If $q_2 < \min\{q_1, p_1\}$, and $p_2 < \min\{q_1, p_1\}$ then \eqref{LM-LM} holds for all $f \in \mathcal{M}^+(\mathbb{R}^n)$ if and only if $C_1 < \infty$, $C_5< \infty$, where  $C_1$ and $C_5$  are defined in  \eqref{C1} and \eqref{C5} respectively, and
\begin{align}\label{C6}
C_6 &:= \Bigg\{\int_0^{\infty} w_1(t)^{q_1} \esup_{ y\in(0, t)} \bigg(\int_y^{\infty} w_1^{q_1}\bigg)^{-1} \Bigg[\int_y^{t} 
\bigg(\int_s^{\infty}w_2^{q_2}\bigg)^{\frac{q_2}{q_1-q_2}} w_2(s)^{q_2}ds \bigg] \notag\\
&  \hspace{1cm} \times \bigg[\int_0^y  \bigg(\int_s^{\infty} w_1^{q_1}\bigg)^{-\frac{q_1}{q_1-p_2}}w_1(s)^{q_1} 
\notag\\
&\hspace{2cm} \times
\bigg(\int_{B(0,y)\setminus B(0,s)} \left[ v_1^{-1} v_2 \right]^{\frac{p_1p_2}{p_1-p_2}}
 \bigg)^{\frac{(p_1-p_2)q_1}{(q_1-p_2)p_1}} ds\bigg]^{\frac{(q_1-p_2)q_2}{(q_1-q_2)p_2}} dt \Bigg\}^{\frac{q_1-q_2}{q_1q_2}}
\end{align}
is finite.
Moreover, 
\[
\|\Id\|_{\LM_{p_1,q_1}(v_1, w_1) \rightarrow \LM_{p_2,q_2}(v_2, w_2)} \approx C_1+C_5+C_6.
\]

{\rm(vii)} If $p_2 <p_1 \leq q_2 < q_1$ then \eqref{LM-LM} holds for all $f \in \mathcal{M}^+(\mathbb{R}^n)$ if and only if $C_1 < \infty$, $C_6< \infty$, where  $C_1$ and $C_6$ are  defined in \eqref{C1} and \eqref{C6} respectively, and 
\begin{align*}
C_7 := \Bigg(\int_0^{\infty}  w_1(t)^{q_1} &\esup_{ y\in(0, t)} \bigg(\int_y^{\infty} w_1^{q_1}\bigg)^{-\frac{q_1}{q_1-q_2}} \esup_{s \in (y,t) }  \bigg( \int_s^{t} w_2^{q_2} \bigg)^{\frac{q_1}{q_1-q_2}}\\
&\times \bigg(\int_{B(0,s)\setminus B(0,y)} \left[ v_1^{-1} v_2 \right]^{\frac{p_1p_2}{p_1-p_2}}
 \bigg)^{\frac{(p_1-p_2)q_1q_2}{(q_1-p_2)p_1p_2}} dt \bigg)^{\frac{q_1-q_2}{q_1q_2}}
\end{align*}
is finite.
Moreover, 
\[
\|\Id\|_{\LM_{p_1,q_1}(v_1, w_1) \rightarrow \LM_{p_2,q_2}(v_2, w_2)} \approx C_1+C_6+C_7.
\]

{\rm(viii)} If $q_1\le p_2 = p_1 \le q_2$ then \eqref{LM-LM} holds for all $f \in \mathcal{M}^+(\mathbb{R}^n)$ if and only if 
\begin{equation}\label{C8}
C_8 := \esup_{t\in (0,\infty)} \bigg(\int_t^{\infty} w_1^{q_1} \bigg)^{-\frac{1}{q_1}} \esup_{s\in (t,\infty)} \bigg( \int_s^{\infty} w_2^{q_2} \bigg)^{\frac{1}{q_2}} 
\esup_{x\in B(0,s)\setminus B(0,t)}  v_1(x)^{-1} v_2(x)  
\end{equation}
is finite.
Moreover, 
\[
\|\Id\|_{\LM_{p_1,q_1}(v_1, w_1) \rightarrow \LM_{p_2,q_2}(v_2, w_2)} \approx C_8.
\]

{\rm(ix)} If $q_1\le q_2 < p_1=p_2$, and $q_1 \le p_2 = p_1$, then \eqref{LM-LM} holds for all $f \in \mathcal{M}^+(\mathbb{R}^n)$ if and only if 
\begin{align}
C_9 &:=  \esup_{t\in (0,\infty)} \bigg(\int_t^{\infty} w_1^{q_1} \bigg)^{-\frac{1}{q_1}}   \notag \\
&\quad \quad \times
\bigg[\int_t^{\infty} \bigg( \int_s^{\infty}w_2^{q_2} \bigg)^{\frac{q_2}{p_1-q_2}} w_2(s)^{q_2}\esup_{x\in B(0,s)\setminus B(0,t)} \left[ v_1(x)^{-1} v_2(x) \right]^{\frac{p_1q_2}{p_1-q_2}}\,ds \bigg]^{\frac{p_1-q_2}{p_1q_2}}\label{C9}
\end{align}
is finite.
Moreover, 
\[
\|\Id\|_{\LM_{p_1,q_1}(v_1, w_1) \rightarrow \LM_{p_2,q_2}(v_2, w_2)} \approx C_9.
\]

{\rm(x)} If $p_1=p_2 < q_1 \le q_2$ then \eqref{LM-LM} holds for all $f \in \mathcal{M}^+(\mathbb{R}^n)$ if and only if $C_8 < \infty$, where $C_8$ is defined in \eqref{C8}, and
\begin{align}
C_{10} := \esup_{t \in (0,\infty)} \bigg(\int_t^{\infty} w_2^{q_2} \bigg)^{\frac{1}{q_2}}  &\bigg(\int_0^{t} \bigg(\int_s^{\infty} w_1^{q_1}\bigg)^{-\frac{q_1}{q_1-p_2}} w_1(s)^{q_1}\notag \\
&\times
\esup_{x\in B(0,t)\setminus B(0,s)} \left[ v_1(x)^{-1} v_2(x) \right]^{\frac{q_1p_2}{q_1-p_2}} ds \bigg)^{\frac{q_1-p_2}{q_1p_2}} \label{C10}
\end{align}
is finite.
Moreover, 
\[
\|\Id\|_{\LM_{p_1,q_1}(v_1, w_1) \rightarrow \LM_{p_2,q_2}(v_2, w_2)} \approx C_8+C_{10}.
\]

{\rm(xi)} If $q_2 < q_1 \leq p_2 = p_1$ then \eqref{LM-LM} holds for all $f \in \mathcal{M}^+(\mathbb{R}^n)$ if and only if 
\begin{align*}
C_{11} :=  \bigg(\int_0^{\infty}  &\bigg(\int_t^{\infty} w_2^{q_2} \bigg)^{\frac{q_2}{q_1-q_2}} w_2(t)^{q_2}  \esup_{s\in (0, t)}  \bigg(\int_s^{\infty} w_1^{q_1}\bigg)^{-\frac{q_2}{q_1-q_2}} \notag \\
&\times
\esup_{x\in B(0,t)\setminus B(0,s)} \left[ v_1(x)^{-1} v_2(x) \right]^{\frac{q_1q_2}{q_1-q_2}} dt \bigg)^{\frac{q_1-q_2}{q_1q_2}} 
\end{align*}
and
\begin{align}
\label{C12}
C_{12}
&:=
\Bigg\{
\int_0^\infty
w_1(t)^{q_1}
\mathop{\operatorname{ess\,sup}}_{y\in(0,t)}
\left(
\int_y^\infty w_1^{q_1}
\right)^{-\frac{q_1}{q_1-q_2}}
\\
\nonumber&\qquad\times
\Bigg(
\int_y^t
\left(
\int_s^t w_2^{q_2}
\right)^{\frac{q_2}{p_1-q_2}}
w_2^{q_2}
\mathop{\operatorname{ess\,sup}}_{x\in B(0,s)\setminus B(0,y)}
\left[
v_1(x)^{-1}v_2(x)
\right]^{\frac{p_1q_2}{p_1-q_2}}
\,ds
\Bigg)^{\frac{(p_1-q_2)q_1}{(q_1-q_2)p_1}}
\,dt
\Bigg\}^{\frac{q_1-q_2}{q_1q_2}} 
\end{align}
are finite.
Moreover, 
\[
\|\Id\|_{\LM_{p_1,q_1}(v_1, w_1) \rightarrow \LM_{p_2,q_2}(v_2, w_2)} \approx C_{11}+C_{12}.
\]

{\rm(xii)} If $q_2 <  p_1 = p_2<  q_1$, then \eqref{LM-LM} holds for all $f \in \mathcal{M}^+(\mathbb{R}^n)$ if and only if $C_8 < \infty$, $C_{12}< \infty$ where  $C_8$ and $C_{12}$  are defined in  \eqref{C8} and \eqref{C12} respectively, and
\begin{align}\label{C13}
C_{13} &:= \bigg(\int_0^{\infty} w_1(t)^{q_1} \esup_{ y\in(0, t)} \bigg(\int_y^{\infty} w_1^{q_1}\bigg)^{-1} \Bigg(\int_y^{t} 
\bigg(\int_{\tau}^{\infty} w_2^{q_2}\bigg)^{\frac{q_2}{q_1-q_2}} w_2(\tau)^{q_2} d\tau \bigg) \\
& \quad \times \bigg(\int_0^y  \bigg(\int_s^{\infty} w_1^{q_1}\bigg)^{-\frac{q_1}{q_1-p_1}}w_1(s)^{q_1} 
\esup_{x\in B(0,y)\setminus B(0,s)} \left[ v_1(x)^{-1} v_2(x) \right]^{\frac{q_1p_1}{q_1-p_1}} ds \bigg)^{\frac{q_2(q_1-p_1)}{p_1(q_1-q_2)}} dt \bigg)^{\frac{q_1-q_2}{q_1q_2}}\notag
\end{align}
is finite.
Moreover, 
\[
\|\Id\|_{\LM_{p_1,q_1}(v_1, w_1) \rightarrow \LM_{p_2,q_2}(v_2, w_2)} \approx C_8+ C_{12}+ C_{13}.
\]

{\rm(xiii)} If $p_2 = p_1 \leq q_2 < q_1$ then \eqref{LM-LM} holds for all $f \in \mathcal{M}^+(\mathbb{R}^n)$ if and only if $C_{8}< \infty$, $C_{13} < \infty$, where  $C_8$ and $C_{13}$ are  defined in  \eqref{C8} and \eqref{C13} respectively, and 
\begin{align*}
C_{14} := \bigg(\int_0^{\infty}  w_1(t)^{q_1} &\esup_{ y\in(0, t)} \bigg(\int_y^{\infty} w_1^{q_1}\bigg)^{-\frac{q_1}{q_1-q_2}} \esup_{s \in (y,t) }  \bigg( \int_s^{t} w_2^{q_2} \bigg)^{\frac{q_1}{q_1-q_2}}\\
&\times 
\esup_{x\in B(0,s)\setminus B(0,y)} \left[ v_1(x)^{-1} v_2(x) \right]^{\frac{q_1q_2}{q_1-q_2}}dt \bigg)^{\frac{q_1-q_2}{q_1q_2}}
\end{align*}
is finite.
Moreover, 
\[
\|\Id\|_{\LM_{p_1,q_1}(v_1, w_1) \rightarrow \LM_{p_2,q_2}(v_2, w_2)} \approx C_8+C_{13}+C_{14}.
\]
\end{theorem}

\begin{proof}
Recall that the embedding \eqref{LM-LM} implies that the identity operator is continuous from $\LM_{p_1,q_1}(v_1, w_1)$ to $\LM_{p_2,q_2}(v_2, w_2)$, which means that the inequality \eqref{Morrey-emb} is satisfied. Also, by Theorem~\ref{T:Morrey-Cesaro}, inequality \eqref{Morrey-emb} holds for all $f\in \mathcal{M}^+(\mathbb{R}^n)$ if and only if \eqref{Cesaro-emb} holds for all $g\in \mathcal{M}^+(0,\infty)$. Note that replacing the entire expression $g^{p_1}$ directly with $h$ does not alter the set of admissible functions in \eqref{Cesaro-emb}, and consequently, inequality \eqref{Cesaro-emb} can be rewritten in the following form:
\begin{align}
\label{eq:260727-1}
\bigg(\int_0^{\infty} \bigg(\int_0^t h(s)^{\frac{p_2}{p_1}} \varphi(s) \, ds\bigg)^{\frac{q_2}{p_2}} w_2(t)^{q_2} dt 
\bigg)^{\frac{p_1}{q_2}} \le \mathcal{C}^{p_1} \bigg( \int_0^{\infty} \bigg(\int_0^t h(s) \, ds \bigg)^{\frac{q_1}{p_1}} 
w_1(t)^{q_1} dt \bigg)^{\frac{p_1}{q_1}}.
\end{align}

In order to apply \cite[Theorem~2.1]{GU-AMP-Ces} to this inequality, we set the parameters as
\begin{equation}\label{pqr-change}
p= \frac{q_1}{p_1}, \quad q= \frac{q_2}{p_1}, \quad r=\frac{p_2}{p_1},
\end{equation}
and define the weight functions
\begin{equation}\label{uvw-change}
v(s)= \varphi(s), \quad u(t)= w_2(t)^{q_2}, \quad w(t)= w_1(t)^{q_1}.
\end{equation}
Since $h$ is arbitrary, it follows that \eqref{eq:260727-1} can be rewritten as
\[
\left(
\int_0^{\infty}
\left(
\int_0^t f(s)^r \varphi(s)\,ds
\right)^{\frac{q}{r}}
u(t)\,dt
\right)^{\frac{p}{q}}
\le
\mathcal{C}^{p_1}
\left(
\int_0^{\infty}
\left(
\int_0^t f(s)\,ds
\right)^p
w(t)\,dt
\right)^{\frac{1}{p}}
\]
for every $f\in\mathcal{M}^+(0,\infty)$.
This is precisely inequality~(1.3) in \cite{GU-AMP-Ces}.

For future reference in characterizing this embedding, we observe the following algebraic relations among these parameters as long as the denominator is not zero:
\begin{align}
\frac{1}{1-r}& = \frac{p_1}{p_1-p_2}, & \frac{r}{1-r} &= \frac{p_2}{p_1-p_2}, & \frac{q}{1-q} &= \frac{q_2}{p_1-q_2}, \nonumber\\
\frac{p}{p-r} &= \frac{q_1}{q_1-p_2}, & \frac{q}{p-q} &= \frac{q_2}{q_1-q_2}, & \frac{p}{p-q} &= \frac{q_1}{q_1-q_2}, \nonumber\\
\frac{pq}{p-q} &=  \frac{q_1q_2}{p_1(q_1-q_2)}, & \frac{pr}{p-r} &= \frac{q_1p_2}{p_1(q_1-p_2)}, & \frac{p(1-q)}{p-q} &= \frac{q_1(p_1-q_2)}{p_1(q_1-q_2)}, \nonumber\\
\frac{q(1-r)}{r(1-q)} &= \frac{q_2(p_1-p_2)}{p_2(p_1-q_2)}, & \frac{p(1-r)}{p-r} &= \frac{q_1(p_1-p_2)}{p_1(q_1-p_2)}, &  \frac{pq(1-r)}{(p-q)r} &= \frac{q_1q_2(p_1-p_2)}{p_1p_2(q_1-q_2)}\nonumber\\
\frac{q(p-r)}{r(p-q)} & = \frac{q_2(q_1-p_2)}{p_2(q_1-q_2)}, & \frac{q(p-1)}{p-q} & = \frac{q_2(q_1-p_1)}{p_1(q_1-q_2)}\label{eq:260727-4}
\end{align}

Recall that the function
$\varphi$ is given by
\eqref{def-varphi}. To characterize the embedding $\LM_{p_1,q_1}(v_1, w_1) \hookrightarrow \LM_{p_2,q_2}(v_2, w_2)$, we rely on the function $V_r(t,s)$ from \cite[Equation~(2.1)]{GU-AMP-Ces}, as recalled below. For the case $p_2 < p_1$ (which corresponds to $r < 1$),  we pass to spherical coordinates to obtain:
\begin{align} \label{eq:260727-5}
V_r(t,s) &= \bigg(\int_{t}^{s} v(\tau)^{\frac{1}{1-r}} \, d\tau \bigg)^{\frac{1-r}{r}} = \bigg(\int_{t}^{s} \varphi(\tau)^{\frac{p_1}{p_1-p_2}} \, d\tau \bigg)^{\frac{p_1-p_2}{p_2}} \nonumber\\ 
& = \bigg(\int_{t}^{s} \left( \int_{S^{n-1}} \left[ v_1(\tau y')^{-1} v_2(\tau y') \right]^{\frac{p_1p_2}{p_1-p_2}} d\sigma(y') \right) \tau^{n-1} \, d\tau \bigg)^{\frac{p_1-p_2}{p_2}} \nonumber\\ 
& = \bigg(\int_{B(0,s)\setminus B(0,t)} \left[ v_1(y)^{-1} v_2(y) \right]^{\frac{p_1p_2}{p_1-p_2}} dy \bigg)^{\frac{p_1-p_2}{p_2}}. 
\end{align}

Similarly, the limiting case $p_1 = p_2$ corresponds to $r=1$. In this case, the functional $V_1(t,s)$ becomes
\begin{align}  \label{eq:260727-6}
V_1(t,s) &= \esup\limits_{\tau \in (t, s)} v(\tau) = \esup\limits_{\tau \in (t, s)} \varphi(\tau) \nonumber\\ 
&= \esup\limits_{\tau \in (t, s)} \esup_{x' \in S^{n-1}} \left[ v_1(\tau x')^{-1} v_2(\tau x') \right]^{p_1}
\nonumber\\ 
&= \esup\limits_{x \in B(0,s)\setminus B(0,t)}  \left[ v_1(x)^{-1} v_2(x) \right]^{p_1}. 
\end{align}
Therefore, by substituting these algebraic relations together with the explicit expressions 
\eqref{eq:260727-5}
and
\eqref{eq:260727-6}
for the functionals
$V_r(t,s)$ (and $V_1(t,s)$) into the corresponding cases of
\cite[Theorem~2.1]{GU-AMP-Ces}, we obtain the constants
$C_1,\dots,C_{14}$.
The table below indicates the correspondence between the cases of
Theorem~3.1 and those of \cite[Theorem~2.1]{GU-AMP-Ces}.
Note that the case $r=1$ cannot occur in
\cite[Theorem~2.1(iv)]{GU-AMP-Ces}.
Remark also that
\[
\frac{pr}{p-r}
=
\frac{p}{p-r}
=
\frac{p}{p-1}
\]
in (viii).
This establishes the required norm equivalences in all cases
(i)--(xiii), thereby completing the proof.
\begin{table}[htbp]
\centering
\begin{tabular}{lll}
\hline
\textbf{Theorem 3.1} &
\textbf{\cite[Theorem~2.1]{GU-AMP-Ces}} &
\textbf{condition} 
 \\
\hline
(i) &
(i) &
$p \le r<1 \le q$\\
(ii) &
(ii) &
$p\le q<1$, $p\le r<1$ \\
(iii) &
(iii) &
$r<p\le q$, $r<1\le q$ \\
(iv) &
(iv) &
$r<p\le q<1$ \\
(v) &
(v) &
$q<p\le r<1$ \\
(vi) &
(vi) &
$q<1$, $q<p$, $r<p$, $r<1$ \\
(vii) &
(vii) &
$r<1\le q<p$ \\
(viii) &
(i) &
$p \le r=1 \le q$ \\
(ix) &
(ii) &
$p\le q<r=1$ \\
(x) &
(iii) &
$r=1<p\le q$ \\
(xi) &
(v) &
$q<p\le r=1$ \\
(xii) &
(vi) &
$q<r=1<p$ \\
(xiii) &
(vii) &
$r=1\le q<p$ \\
* &
$r=1<p\le q<1$ &
impossible \\
\hline
\end{tabular}

\bigskip

\noindent
\end{table}
\end{proof}
\newpage

\subsection{\texorpdfstring{$\dual \LM_{p_1,q_1}(v_1, w_1) \hookrightarrow \LM_{p_2,q_2}(v_2, w_2)$}{dual LM(p1,q1) to LM(p2,q2)}}

Denote by ${\Id}$
the embedding operator
if it exists and define
\[\|\Id\|_{\dual \LM_{p_1,q_1}(v_1, w_1) \rightarrow \LM_{p_2,q_2}(v_2, w_2)}=\sup_{f\in \dual \LM_{p_1,q_1}(v_1, w_1)\setminus \{0\} }\frac{\|f\|_{ \LM_{p_2,q_2}(v_2, w_2)}}{ \|f\|_{\dual \LM_{p_1,q_1}(v_1, w_1)}}.\]

\begin{theorem}\label{T:cLM-LM}
Let $0 < p_1, p_2, q_1, q_2 < \infty$. Assume that $w_1$ and $w_2$ are weights on $(0, \infty)$ such that $\int_0^t w_1^{q_1} < \infty$ and $ \int_t^{\infty} w_2^{q_2} < \infty$ for all $t \in (0, \infty)$ and $v_1$ and $v_2$ are  weights on $\mathbb{R}^n$.

{\rm(i)} If  $q_1\le p_2< p_1 \le q_2$, then \eqref{cLM-LM} holds for all $f \in \mathcal{M}^+(\mathbb{R}^n)$ if and only if 
\begin{equation}\label{C1-2}
\mathcal{C}_1 :=  \sup_{t\in (0,\infty)} \bigg(\int_0^t w_1^{q_1} \bigg)^{-\frac{1}{q_1}}  \bigg( \int_t^{\infty} w_2^{q_2} \bigg)^{\frac{1}{q_2}} \bigg(\int_{B(0,t)} \left[ v_1^{-1} v_2 \right]^{\frac{p_1p_2}{p_1-p_2}} \bigg)^{\frac{p_1-p_2}{p_1p_2}} 
\end{equation}
is finite.
Moreover, 
\[
\|\Id\|_{\dual \LM_{p_1,q_1}(v_1, w_1) \rightarrow \LM_{p_2,q_2}(v_2, w_2)} \approx \mathcal{C}_1.
\]

{\rm(ii)} If  $q_1\le q_2 < p_1$, and $q_1 \le p_2 < p_1$ then \eqref{cLM-LM} holds for all $f \in \mathcal{M}^+(\mathbb{R}^n)$ if and only if $\mathcal{C}_1 <\infty$, where $\mathcal{C}_1$ is defined in \eqref{C1-2} and
\begin{align}
\mathcal{C}_2 :=  \sup_{t\in (0,\infty)} &\bigg(\int_0^{t} w_1^{q_1} \bigg)^{-\frac{1}{q_1}}  \bigg( \int_0^{t} \bigg( \int_s^{\infty} w_2^{q_2} \bigg)^{\frac{q_2}{p_1-q_2}} w_2(s)^{q_2}\notag \\
&\times \bigg(\int_{B(0,s)} \left[ v_1^{-1} v_2 \right]^{\frac{p_1p_2}{p_1-p_2}} \bigg)^{\frac{(p_1-p_2)q_2}{(p_1-q_2)p_2}}\,ds \bigg)^{\frac{p_1-q_2}{p_1q_2}} 
\label{C2-2}
\end{align}
is finite.
Moreover, 
\[
\|\Id\|_{\dual \LM_{p_1,q_1}(v_1, w_1) \rightarrow \LM_{p_2,q_2}(v_2, w_2)} \approx \mathcal{C}_1+ \mathcal{C}_2.
\]

{\rm(iii)} If  $p_2 < q_1 \le q_2$, and $p_2 < p_1 \le q_2$ then \eqref{cLM-LM} holds for all $f \in \mathcal{M}^+(\mathbb{R}^n)$ if and only if $\mathcal{C}_1 < \infty$, where $\mathcal{C}_1$ is defined in \eqref{C1-2} and
\begin{align}
\mathcal{C}_3 :=  \sup_{t \in (0,\infty)} &\bigg(\int_t^{\infty} w_2^{q_2} \bigg)^{\frac{1}{q_2}}  \bigg(\int_0^{t} \bigg(\int_0^s w_1^{q_1}\bigg)^{-\frac{q_1}{q_1-p_2}} w_1(s)^{q_1}\notag \\
&\times\bigg(\int_{B(0,s)} \left[ v_1^{-1} v_2 \right]^{\frac{p_1p_2}{p_1-p_2}} \bigg)^{\frac{(p_1-p_2)q_1}{(q_1-p_2)p_1}}\,ds  \bigg)^{\frac{q_1-p_2}{q_1p_2}}  \label{C3-2}
\end{align}
is finite.
Moreover, 
\[
\|\Id\|_{\dual \LM_{p_1,q_1}(v_1, w_1) \rightarrow \LM_{p_2,q_2}(v_2, w_2)} \approx \mathcal{C}_1 + \mathcal{C}_3.
\]

{\rm(iv)} If $p_2 < q_1 \le q_2 < p_1$
then \eqref{cLM-LM} holds for all $f \in \mathcal{M}^+(\mathbb{R}^n)$ if and only if $\mathcal{C}_1<\infty$,
$\mathcal{C}_2 < \infty$ and $\mathcal{C}_3< \infty$, where $\mathcal{C}_1$, $\mathcal{C}_2$ and $\mathcal{C}_3$ are defined in \eqref{C1-2}, \eqref{C2-2} and \eqref{C3-2}, respectively. Moreover, 
\[
\|\Id\|_{\dual \LM_{p_1,q_1}(v_1, w_1) \rightarrow \LM_{p_2,q_2}(v_2, w_2)} \approx \mathcal{C}_1 + \mathcal{C}_2 + \mathcal{C}_3.
\]

{\rm(v)} If $q_2 < q_1 \leq p_2 < p_1$ then \eqref{cLM-LM} holds for all $f \in \mathcal{M}^+(\mathbb{R}^n)$ if and only if 
\begin{align*}
\mathcal{C}_4 := \bigg\{\int_0^{\infty}  & \bigg(\int_t^{\infty} w_2^{q_2} \bigg)^{\frac{q_2}{q_1-q_2}} w_2(t)^{q_2}  \notag \\
&\times \esup_{s\in (0, t)} \left[ \bigg(\int_0^s w_1^{q_1}\bigg)^{-\frac{q_2}{q_1-q_2}} \bigg(\int_{B(0,s)} \left[ v_1^{-1} v_2 \right]^{\frac{p_1p_2}{p_1-p_2}} \bigg)^{\frac{(p_1-p_2)q_1q_2}{(q_1-q_2)p_1p_2}}
\right]\,dt  \bigg\}^{\frac{q_1-q_2}{q_1q_2}} 
\end{align*}
and
\begin{align}\label{C5-2}
\mathcal{C}_5 & :=  \bigg\{\int_0^{\infty} \bigg(\int_0^t w_1^{q_1} \bigg)^{-\frac{q_1}{q_1-q_2}} w_1(t)^{q_1}  \notag\\ 
& \hspace{1cm} \times \bigg[ \int_0^{t} \bigg( \int_s^{t} w_2^{q_2} \bigg)^{\frac{q_2}{p_1-q_2}} w_2(s)^{q_2}\bigg(\int_{B(0,s)} \left[ v_1^{-1} v_2 \right]^{\frac{p_1p_2}{p_1-p_2}} \bigg)^{\frac{(p_1-p_2)q_2}{(p_1-q_2)p_2}} ds \bigg]^{\frac{q_1(p_1-q_2)}{p_1(q_1-q_2)}} dt \bigg\}^{\frac{q_1-q_2}{q_1q_2}} \notag \\
& \hspace{0.5cm}+  \bigg(\int_0^{\infty}  w_1^{q_1} \bigg)^{-\frac{1}{q_1}} \bigg\{ \int_0^{\infty} \bigg( \int_t^{\infty} w_2^{q_2} \bigg)^{\frac{q_2}{p_1-q_2}} w_2(t)^{q_2} \bigg(\int_{B(0,t)} \left[ v_1^{-1} v_2 \right]^{\frac{p_1p_2}{p_1-p_2}} \bigg)^{\frac{(p_1-p_2)q_2}{(p_1-q_2)p_2}} dt \bigg\}^{\frac{p_1-q_2}{p_1q_2}}
\end{align}
are finite.
Moreover, 
\[
\|\Id\|_{\dual \LM_{p_1,q_1}(v_1, w_1) \rightarrow \LM_{p_2,q_2}(v_2, w_2)} \approx \mathcal{C}_4 + \mathcal{C}_5.
\]

{\rm(vi)} If $q_2 < \min\{q_1, p_1\}$, and $p_2 < \min\{q_1, p_1\}$ then \eqref{cLM-LM} holds for all $f \in \mathcal{M}^+(\mathbb{R}^n)$ if and only if $\mathcal{C}_5< \infty$, where $\mathcal{C}_5$  is defined in \eqref{C5-2}, and
\begin{align}\label{C6-2}
\mathcal{C}_6 & :=  \bigg\{\int_0^{\infty} \bigg(\int_0^t w_1^{q_1} \bigg)^{-2} w_1(t)^{q_1}  \sup_{y \in (0, t)} \bigg(\int_0^y w_1^{q_1} \bigg) \bigg[\int_y^t w_2(s)^{q_2} \bigg(\int_s^{\infty} w_2^{q_2} \bigg)^{\frac{q_2}{q_1-q_2}} ds\bigg]  \\
& \hspace{0.7cm}\times\bigg[\int_0^y \bigg(\int_0^s w_1^{q_1} \bigg)^{-\frac{q_1}{q_1-p_2}} w_1(s)^{q_1} \bigg(\int_{B(0,s)} \left[ v_1^{-1} v_2 \right]^{\frac{p_1p_2}{p_1-p_2}}
\bigg)^{\frac{(p_1-p_2)q_1}{(q_1-p_2)p_1}} ds \bigg]^{\frac{q_2(q_1-p_2)}{p_2(q_1-q_2)}} dt \bigg\}^{\frac{q_1-q_2}{q_1q_2}}
\notag
\end{align}
is finite.
Moreover, 
\[
\|\Id\|_{\dual \LM_{p_1,q_1}(v_1, w_1) \rightarrow \LM_{p_2,q_2}(v_2, w_2)} \approx \mathcal{C}_5 + \mathcal{C}_6.
\]

{\rm(vii)} If $p_2 <p_1 \leq q_2 < q_1$ then \eqref{cLM-LM} holds for all $f \in \mathcal{M}^+(\mathbb{R}^n)$ if and only if $\mathcal{C}_6< \infty$, where $\mathcal{C}_6$ is  defined in \eqref{C6-2}, and 
\begin{align*}
\mathcal{C}_7 & := \bigg\{\int_0^{\infty} \bigg( \int_0^t w_1^{q_1} \bigg)^{-\frac{q_1}{q_1-q_2}} w_1(t)^{q_1} \\
& \hspace{1.5cm} \times
\bigg[\esup_{s \in (0,t)} \bigg( \int_s^{\infty} w_2^{q_2} \bigg)^{\frac{q_1}{q_1-q_2}} \bigg(\int_{B(0,s)} \left[ v_1^{-1} v_2 \right]^{\frac{p_1p_2}{p_1-p_2}} \bigg)^{\frac{(p_1-p_2)q_1q_2}{(q_1-p_2)p_1p_2}}
\bigg]dt \bigg\}^{\frac{q_1-q_2}{q_1q_2}}
\notag\\
& \hspace{0.5cm}+ \bigg(\int_0^{\infty}  w_1^{q_1} \bigg)^{-\frac{1}{q_1}} \esup_{t \in (0, {\infty})} \bigg( \int_t^{\infty} w_2^{q_2} \bigg)^{\frac{1}{q_2}} \bigg(\int_{B(0,t)} \left[ v_1^{-1} v_2 \right]^{\frac{p_1p_2}{p_1-p_2}} \bigg)^{\frac{p_1-p_2}{p_1p_2}}
\end{align*}
is finite. Moreover, 
\[
\|\Id\|_{\dual \LM_{p_1,q_1}(v_1, w_1) \rightarrow \LM_{p_2,q_2}(v_2, w_2)} \approx \mathcal{C}_6 + \mathcal{C}_7.
\]

{\rm(viii)} If $q_1\le p_2 = p_1 \le q_2$ then \eqref{cLM-LM} holds for all $f \in \mathcal{M}^+(\mathbb{R}^n)$ if and only if 
\begin{equation}\label{C8-2}
\mathcal{C}_8 :=  
\esup_{t\in (0,\infty)} \bigg(\int_0^t w_1^{q_1} \bigg)^{-\frac{1}{q_1}}  \bigg( \int_t^{\infty} w_2^{q_2} \bigg)^{\frac{1}{q_2}} 
\esup_{x\in B(0,t)} v_1(x)^{-1} v_2(x)  
\end{equation}
is finite.
Moreover, 
\[
\|\Id\|_{\dual \LM_{p_1,q_1}(v_1, w_1) \rightarrow \LM_{p_2,q_2}(v_2, w_2)} \approx \mathcal{C}_8.
\]

{\rm(ix)} If $q_1\le q_2 < p_1=p_2$, and $q_1 \le p_2 = p_1$ then \eqref{cLM-LM} holds for all $f \in \mathcal{M}^+(\mathbb{R}^n)$ if and only if $\mathcal{C}_8<\infty$, where $\mathcal{C}_8$ is defined in \eqref{C8-2}, and
\begin{align}
\mathcal{C}_9 &:=  \sup_{t\in (0,\infty)} \bigg(\int_0^{t} w_1^{q_1} \bigg)^{-\frac{1}{q_1}}   \notag \\
&\hspace{2cm} \times\bigg( \int_0^{t} \bigg( \int_s^{\infty} w_2^{q_2} \bigg)^{\frac{q_2}{p_1-q_2}} w_2(s)^{q_2} \esup_{x\in B(0,s)} \left[ v_1(x)^{-1} v_2(x) \right]^{\frac{p_1q_2}{p_1-q_2}}
\,ds \bigg)^{\frac{p_1-q_2}{p_1q_2}}  \label{C9-2}
\end{align}
is finite.
Moreover, 
\[
\|\Id\|_{\dual \LM_{p_1,q_1}(v_1, w_1) \rightarrow \LM_{p_2,q_2}(v_2, w_2)} \approx \mathcal{C}_8 + \mathcal{C}_9.
\]

{\rm(x)} If $p_1=p_2 < q_1 \le q_2$ then \eqref{cLM-LM} holds for all $f \in \mathcal{M}^+(\mathbb{R}^n)$ if and only if $\mathcal{C}_8 < \infty$, where $\mathcal{C}_8$ is defined in \eqref{C8-2} and
\begin{align}
\mathcal{C}_{10} &:= \sup_{t \in (0,\infty)} \bigg(\int_t^{\infty} w_2^{q_2} \bigg)^{\frac{1}{q_2}}  \notag \\
& \times\bigg\{\int_0^{t} \bigg(\int_0^s w_1^{q_1} \bigg)^{-\frac{q_1}{q_1-p_2}} w_1(s)^{q_1} \esup_{x\in B(0,s)} \left[ v_1(x)^{-1} v_2(x) \right]^{\frac{q_1p_2}{q_1-p_2}} ds \bigg\}^{\frac{q_1-p_2}{q_1p_2}} \label{C10-2}
\end{align}
is finite.
Moreover, 
\[
\|\Id\|_{\dual \LM_{p_1,q_1}(v_1, w_1) \rightarrow \LM_{p_2,q_2}(v_2, w_2)} \approx \mathcal{C}_8 + \mathcal{C}_{10}.
\]

{\rm(xi)} If $q_2 < q_1 \leq p_2 = p_1$ then \eqref{cLM-LM} holds for all $f \in \mathcal{M}^+(\mathbb{R}^n)$ if and only if 
\begin{align*}
\mathcal{C}_{11} :=  \bigg(\int_0^{\infty}  & \bigg(\int_t^{\infty} w_2^{q_2} \bigg)^{\frac{q_2}{q_1-q_2}} w_2(t)^{q_2}  \esup_{s\in (0, t)}  \bigg(\int_0^s w_1^{q_1} \bigg)^{-\frac{q_2}{q_1-q_2}} \notag \\
&\times \esup_{x\in B(0,s)} \left[ v_1(x)^{-1} v_2(x) \right]^{\frac{q_1q_2}{q_1-q_2}} dt \bigg)^{\frac{q_1-q_2}{q_1q_2}} 
\end{align*}
and
\begin{align} \label{C12-2}
\mathcal{C}_{12} & :=  \bigg(\int_0^{\infty} \bigg(\int_0^t w_1^{q_1} \bigg)^{-\frac{q_1}{q_1-q_2}} w_1(t)^{q_1} \bigg( \int_0^{t} \bigg( \int_s^{t} w_2^{q_2} \bigg)^{\frac{q_2}{p_1-q_2}} w_2(s)^{q_2} \notag\\
& \hspace{3cm} \times \esup_{x\in B(0,s)} \left[ v_1(x)^{-1} v_2(x) \right]^{\frac{p_1q_2}{p_1-q_2}} ds \bigg)^{\frac{q_1(p_1-q_2)}{p_1(q_1-q_2)}} dt \bigg)^{\frac{q_1-q_2}{q_1q_2}} \\
& \hspace{0.5cm} +   \bigg(\int_0^{\infty}  w_1^{q_1} \bigg)^{-\frac{1}{q_1}} \bigg( \int_0^{\infty} \bigg( \int_t^{\infty} w_2^{q_2} \bigg)^{\frac{q_2}{p_1-q_2}} w_2(t)^{q_2} \esup_{x\in B(0,t)} \left[ v_1(x)^{-1} v_2(x) \right]^{\frac{p_1q_2}{p_1-q_2}} dt \bigg)^{\frac{p_1-q_2}{p_1q_2}}\notag
\end{align}
are finite.
Moreover, 
\[
\|\Id\|_{\dual \LM_{p_1,q_1}(v_1, w_1) \rightarrow \LM_{p_2,q_2}(v_2, w_2)} \approx \mathcal{C}_{11} + \mathcal{C}_{12}.
\]

{\rm(xii)} If $q_2 <  p_1 = p_2<  q_1$, then \eqref{cLM-LM} holds for all $f \in \mathcal{M}^+(\mathbb{R}^n)$ if and only if $\mathcal{C}_{12}< \infty$ where  $\mathcal{C}_{12}$  is defined in  \eqref{C12-2}, and
\begin{align}\label{C13-2}
\mathcal{C}_{13}  & :=  \bigg(\int_0^{\infty} \bigg(\int_0^t w_1^{q_1} \bigg)^{-2} w_1(t)^{q_1}  \sup_{y \in (0, t)} \bigg(\int_0^y w_1^{q_1} \bigg) \bigg(\int_y^t w_2(s)^{q_2} \bigg(\int_s^{\infty} w_2^{q_2} \bigg)^{\frac{q_2}{q_1-q_2}}ds \bigg) \notag \\
& \hspace{1cm}\times \bigg(\int_0^y \bigg(\int_0^s w_1^{q_1} \bigg)^{-\frac{q_1}{q_1-p_2}} w_1(s)^{q_1} \esup_{x\in B(0,s)} \left[ v_1(x)^{-1} v_2(x) \right]^{\frac{q_1p_2}{q_1-p_2}}  ds \bigg)^{\frac{q_2(q_1-p_2)}{p_2(q_1-q_2)}}  dt \bigg)^{\frac{q_1-q_2}{q_1q_2}}
\end{align}
is finite.
Moreover, 
\[
\|\Id\|_{\dual \LM_{p_1,q_1}(v_1, w_1) \rightarrow \LM_{p_2,q_2}(v_2, w_2)} \approx \mathcal{C}_{12}+ \mathcal{C}_{13}.
\]

{\rm(xiii)} If $p_2 = p_1 \leq q_2 < q_1$ then \eqref{cLM-LM} holds for all $f \in \mathcal{M}^+(\mathbb{R}^n)$ if and only if $\mathcal{C}_{13} < \infty$, where  $\mathcal{C}_{13}$ is  defined in  \eqref{C13-2}, and 
\begin{align*}
\mathcal{C}_{14} & := \bigg(\int_0^{\infty} \bigg(\int_0^t w_1^{q_1} \bigg)^{-\frac{q_1}{q_1-q_2}} w_1(t)^{q_1} \\
& \hspace{2cm} \times \esup_{s \in (0,t)} \bigg( \int_s^{\infty} w_2^{q_2} \bigg)^{\frac{q_1}{q_1-q_2}} \esup_{x\in B(0,t)} \left[ v_1(x)^{-1} v_2(x) \right]^{\frac{q_1q_2}{q_1-q_2}} dt \bigg)^{\frac{q_1-q_2}{q_1q_2}}
\notag\\
& \hspace{1cm}+ \bigg(\int_0^{\infty}  w_1^{q_1} \bigg)^{-\frac{1}{q_1}} \esup_{t \in (0, {\infty})} \bigg( \int_t^{\infty} w_2^{q_2} \bigg)^{\frac{1}{q_2}} \esup_{x\in B(0,t)} v_1(x)^{-1} v_2(x)
\end{align*}
is finite. Moreover, 
\[
\|\Id\|_{\dual \LM_{p_1,q_1}(v_1, w_1) \rightarrow \LM_{p_2,q_2}(v_2, w_2)} \approx \mathcal{C}_{13} + \mathcal{C}_{14}.
\]
\end{theorem}

\begin{proof}
The proof follows a similar strategy to that of Theorem~\ref{T:LM-LM}. First, we recall that the embedding \eqref{cLM-LM} implies that the identity operator is bounded from $\dual\LM_{p_1,q_1}(v_1, w_1)$ to $\LM_{p_2,q_2}(v_2, w_2)$.
Recall again that the function
$\varphi$ is given by
\eqref{def-varphi}.

By utilizing Theorem~\ref{T:cLM-LM-Copson-Cesaro},  \eqref{cLM-LM-emb} holds for all $f \in \mathcal{M}^+(\mathbb{R}^n)$ if and only if \eqref{Copson-Cesaro-emb} holds for all $g \in \mathcal{M}^+(0,\infty)$. Following a similar substitution to the previous proof, we can express this equivalent inequality in the following form suitable for applying known weight characterizations:
\begin{align}\label{eq:260727-3}
\bigg(\int_0^{\infty} \bigg(\int_0^t h(s)^{\frac{p_2}{p_1}} \varphi(s) \, ds\bigg)^{\frac{q_2}{p_2}} w_2(t)^{q_2} dt 
\bigg)^{\frac{p_1}{q_2}} \le \mathcal{C}^{p_1} \bigg( \int_0^{\infty} \bigg(\int_t^{\infty} h(s) \, ds \bigg)^{\frac{q_1}{p_1}} 
w_1(t)^{q_1} dt \bigg)^{\frac{p_1}{q_1}}.
\end{align}

To evaluate this inequality, we invoke \cite[Theorem~2.1]{GPU-JFA}. We set the parameters exactly as in \eqref{pqr-change}:
\[
p= \frac{q_1}{p_1}, \quad q= \frac{q_2}{p_1}, \quad r=\frac{p_2}{p_1},
\]
and define the associated weight functions exactly as in \eqref{uvw-change}:
\[
v(s)= \varphi(s), \quad u(t)= w_2(t)^{q_2}, \quad w(t)= w_1(t)^{q_1}.
\]
So \eqref{eq:260727-3} reads
\[
\bigg(\int_0^{\infty} \bigg(\int_0^t f(s)^r \varphi(s) \, ds\bigg)^{\frac{q}{r}} u(t) dt 
\bigg)^{\frac{1}{q}} \le \mathcal{C}^{p_1} \bigg( \int_0^{\infty} \bigg(\int_t^{\infty} f(s) \, ds \bigg)^{p} 
w(t) dt \bigg)^{\frac{1}{p}}.
\]
The algebraic relations among the parameters $p, q,$ and $r$ remain identical to those listed in the proof of Theorem~\ref{T:LM-LM}.
See \eqref{eq:260727-4}.

The primary distinction in this case lies in the functional $V_r$ (and its limiting case $V_1$), which arises from the integration domains of the operators considered in \cite[Theorem~2.1]{GPU-JFA}. Unlike \cite[Theorem~2.1]{GPU-JFA}, where the functional $V_r(t,s)$ is used, in Theorem~\ref{T:cLM-LM}, we require only the special case $V_r(0,s)$. For the case $p_2<p_1$ (equivalently, $r<1$), passing to spherical coordinates yields a functional involving integration over the entire ball $B(0,s)$, rather than over an annulus. Indeed, setting $t=0$ in \eqref{eq:260727-5} gives
\begin{align*} 
V_r(0,s) & = \bigg(\int_{B(0,s)} \left[ v_1(x)^{-1} v_2(x) \right]^{\frac{p_1p_2}{p_1-p_2}} dx \bigg)^{\frac{p_1-p_2}{p_2}}. 
\end{align*}

Similarly, for the limiting case $p_1 = p_2$ (corresponding to $r=1$), the functional evaluates the essential supremum over the ball $B(0,s)$:
Letting $t=0$ in \eqref{eq:260727-6},
we obtain
\begin{align*} 
V_1(0,s) 
&= \esup\limits_{x \in B(0,s)}  \left[ v_1(x)^{-1} v_2(x) \right]^{p_1}. 
\end{align*}

Substituting these algebraic relations, the weight functions $u, v, w$, and the explicit expressions for $V_r(0,s)$ and $V_1(0,s)$ into the respective cases of \cite[Theorem~2.1]{GPU-JFA} generates the conditions $\mathcal{C}_1$ through $\mathcal{C}_{14}$. This establishes the required norm equivalences for all cases (i) through (xiii) and completes the proof.
\end{proof}

\section*{Acknowledgments}
The research of Gogatishvili was supported by the Institute of Mathematics of the Czech Academy of Sciences by RVO: 67985840. and  by the
grant Ministry of Education and Science of the Republic of Kazakhstan (project no. AP26196065)


\end{document}